\documentclass[12pt,article,noinfoline]{imsart}
\usepackage{graphicx,enumerate}

\usepackage{jr2,comment}

\newtheorem{theorem}{Theorem}
\newtheorem{lemma}{Lemma}
\newtheorem{Assumption}{Assumption}

\newtheorem{example}{Example}
\startlocaldefs
\def\bb#1{\mathbb{#1}}

\def\boldsymbol#1{\setbox\ewb\hbox{$#1$}%
    \setlength{\deno}{-\wd\ewb+0.05em}{ #1}\hspace{\deno}{#1}}
\endlocaldefs

\begin{document}

\baselineskip=23pt

\title{Consistent Empirical Bayes Estimation  of the Mean of a Mixing Distribution with Applications to
Treatment of Nonresponse}

\begin{center}
{ {\bf \large  Eitan Greenshtein}

Central Bureau of Statistics, Israel

eitan.greenshtein@gmail.com
} 
\end{center}
{\bf Abstract}

We consider a Nonparametric  Empirical Bayes  (NPEB) framework. Let  $Y_i$ be random variables, $Y_i \sim f(y|\theta_i)$, $i=1,...,n$, where
$\theta_i \sim G$, and $\theta_i \in \Theta$ are independent. The variables $Y_i $ are conditionally  independent
given $\theta_i, \; i=1,...,n$.
The mixing distribution  $G$ is unknown   and  assumed to belong  to a nonparametric class
 $\{G \}$.

Let $\eta(\theta)$  be a function of $\theta$. We address the problem of consistently estimating  $E_G \eta(\theta) \equiv \eta_G$. This problem becomes
particularly challenging when $G$ cannot be consistently estimated from the observed data.

We motivate this problem, especially in contexts involving nonresponse and missing data. 
For such  cases, a consistent estimation method is suggested and its performance is demonstrated  through simulations.

\bigskip
 
\section{Introduction}

Our goal is to utilize Nonparametric  Empirical Bayes (NPEB) methods to estimate  the means of mixtures,  particularly in contexts involving sampling with nonresponse.
Let $\theta_i \sim G$, $i=1,...,n,$ be i.i.d., let  $Y_i \sim f(y|\theta_i), \; \theta_i \in \Theta$. Here $f(y|\theta_i) $  is the density of the distribution of $Y_i$ under $\theta_i$  with respect to  some dominating  measure $\mu$.
The variables $Y_i$ are independent conditional on $\theta_i$, for $i=1,...,n$.

Let $$\eta_{G}= E_G \eta(\theta).$$

Given an estimator $\hat{G}$ for $G$, a natural  estimator  for $\eta_G$,  is:
$$ \eta_{\hat{G}}= E_{\hat{G}} \eta(\theta).$$

Given observations $t(Y_1),...,t(Y_n)$, our default estimator $\hat{G} \equiv \hat{G}^n$ for $G$ 
is a Generalized Maximum Likelihood Estimator (GMLE), also referred to as 
 Nonparametric Maximum Likelihood Estimator (NPMLE), as  proposed by Kiefer
and Wolfowitz (1956) (see definition in the sequel). Here we write $t(Y_i)$ since we do not necessarily observe
$Y_i$, e.g., we observe  a truncated or censored version of $Y_i$. To simplify notations, especially
when we present the general theory, we will often just write $Y_i$ when there is no confusion.

 In what follows, consistency  of  a sequence of estimators $\hat{G}^n$ for $G$, is understood in the sense of  almost sure (a.s.) weak convergence;
we write $\hat{G}^n \Rightarrow G$.
In many cases, no consistent sequence  of estimators  $\hat{G}^n$ for $G$ exists, and the task of consistently estimating
$\eta_G$ appears infeasible.
 
However, if there were a sequence of estimators $\hat{G}^n$ that converged weakly to $G$ almost surely, then the
corresponding  sequence
of estimators $\eta_{\hat{G}}$ would be consistent for $\eta_G$, provided $\eta(\theta)$ is continuous
and bounded. Specifically, $\hat{G}^n \Rightarrow G$, implies
 $$\eta_{\hat{G}^n} \rightarrow_{a.s.} \eta_G.$$ 

\bigskip


\bigskip

In the following simple example, we demonstrate the  difficulty  and a possible remedy for the 
nonexistence of a definite  sequence $\hat{G}^n$ that converges weakly to $G$.

\begin{example} \label{ex:first}
Let  $\theta_i \sim G$ be i.i.d, $Y_i \sim Bernoulli(\theta_i)$, $\theta_i \in [0,1]$, $i=1,...,n$.  Suppose
the order statistics of the observations is $Y_{(1)}=...=Y_{(n/2)}=0, Y_{(n/2+1)}=...= Y_{(n)}=1$.

It can be shown that  both the following $\hat{G}_1$ and  $\hat{G}_2$ are GMLE (to be defined)
for this dataset. 

$\bullet$ Let $\hat{G}_1$ satisfy $P_{\hat{G}_1}(\theta=0.5)=1$. 

$\bullet$ Let  $\hat{G}_2$ satisfy $P_{\hat{G}_2}(\theta=0.75)=P_{\hat{G}_2}(\theta=0.25)=0.5.$

For the  function $\eta(\theta)=\theta^2$, we find that:
$$ E_{\hat{G}_1}\eta(\theta)=0.25 \neq E_{\hat{G}_2}\eta(\theta)=0.312.$$ 

However, for  the function $\eta(\theta)=\theta$, we have: 
$$ E_{\hat{G}_1} \eta(\theta)= E_{\hat{G}_2} \eta(\theta)=0.5. $$
\end{example}

A key difference between the above two functions $\eta(\theta)$  is that
in the case $\eta(\theta)=\theta$ we have  $\eta(\theta)= E_{\theta} h(Y)$ for $h(Y)=Y$.
Hence  $E_G\eta(\theta)=E_{G}  h(Y)$, in which case a natural unbiased estimator of $\eta_G$
is  
\begin{equation}
 \frac{1}{n} \sum_i h(Y_i).  \label{eqn:missing}
\end{equation}
Thus, in situations where $\eta_G=E_G h(Y)$, we   have the above natural estimator, which does not involve
estimating $G$. Greenshtein and Ritov (2022)  show that the above estimator and the estimator $\eta_{\hat{G}}$ are identical
in a more general setup.

The main focus of this paper is the estimation of $$E_G \eta(\theta)=E_{G}  h(Y)$$ in situations
where there are Missing Not At Random (MNAR) observations among $Y_i, i=1,...,n$. 
 In such cases the naive estimator (\ref{eqn:missing}) 
(applied to the remaining observations)
may be inconsistent and unreliable.
We  demonstrate that estimators of the form $$E_{\hat{G}}  \eta(\theta) = \eta_{\hat{G}}$$ often
remain consistent, even when $G$ cannot be consistently estimated, and there are MNAR observations $Y_i$.

The problem of estimating  $\eta_G$, the mean of $\eta(\theta)$ under the mixture $G$, has scarcely  been studied in the 
Empirical Bayes literature. It has been examined in 
Greenshtein and Itskov (2018) and in Greenshtein and  Ritov (2022). In those papers,
we were intrigued by the surprisingly strong performance of 
the estimator $\eta_{\hat{G}}$ in situations where $G$ cannot be consistently estimated and is 
nonidentifiable. Understanding this phenomenon  motivates   the present study.

A related line of research exists in econometrics;  see, e.g.,  Dobronyi, Gu, and Kim (2021). In their work, the goal is to estimate functionals of a latent
distribution of latent  variables (analogous to our mixture distribution of the parameters, 
which are `latent variables'), in a dynamic  panel logit model. Their approach  involves  estimating moments of the latent distribution, and leveraging the relationship between moments of a distribution and the distribution itself (the moment problem).  They also report  surprisingly good  performance of estimates of
functionals of the unknown distribution of  the latent variables, in cases where that distribution is nonidentifiable from the observed data.

In cases  with  MNAR observations, the available data  are typically  not representative of the population due to a selection bias. This phenomenon also occurs in observational studies.
The NPEB approach in the current paper is designed  to correct such   bias. The papers  by Eckles, Ignatiadis, et al.  (2025), and by Robbins and Zhang (1991) (see also references therein), apply Empirical Bayes ideas to handle  selection bias due to ``Regression Discontinuity Designs". As an example, consider a clinical trial where it is desired
to estimate the treatment effect of a new drug. However, suppose that only ``difficult cases" are assigned to the treatment-group.  

Let $I_i$ be an indicator of the event ``a response was obtained from individual $i$''.
 In this paper, similarly to the previously mentioned papers, given a particular function $h$, a main idea is the following. Conditional on the
unobserved parameters $\theta_i$,  it is assumed that $E(h(Y_i)|I_i)$ and $I_i$ are independent. 

The paper of Vardi (1985) has a similar motivation of handling selection bias, see also Turnbull (1976).
However, in both of those papers the goal is to estimate the marginal distribution of $Y$, while our 
more modest goal
 is to estimate  certain functionals of that distribution.
The former handles selection bias   in a general context, the latter is motivated by interval censoring in survival
analysis. Both, essentially use GMLE. We consider a selection bias that is due to non-response. The parametrization and  setup in
the present paper differs substantially.

 The
additional main point established in the present  paper, is that although $G$ is non-identifiable, a meaningful inference may still be conducted.

\bigskip

{\bf Basic Definitions}


Given a mixing  distribution $G$ and a dominated family of distributions with densities {\nolinebreak$\{ f(y\mid \theta):\; \theta \in \Theta \}$} with respect to some dominating measure $\mu$, define
$$ f_G(y)= \int f(y\mid \vartheta) dG(\vartheta).$$

{\bf GMLE ( Generalized Maximum Likelihood  Estimator)}. Given observations $Y_1,...,Y_n$,
a {\it GMLE} $\hat{G}$ for $G$ (Kiefer and Wolfowitz, 1956) is defined  as any $\hat{G}$ satisfying:
\begin{equation} \label{eq:hatG} \hat{G}= \argmax_{G \in \{G\}} \; \Pi_{i=1}^n f_{{G}}(Y_i). \end{equation}
The maximization is taken  over  all probability distributions $\{G\}$ on $\Theta$.

{\bf Identifiability.} Under the above setup, we say that the model has  {\it identifiability}, if for every
$G_1 \in \{G\}$  and $G_2 \in \{G\}$, $$f_{G_1}=f_{G_2} \Rightarrow \; G_1=G_2.$$

Revisiting  Example \ref{ex:first}. In light of these definitions, it may be seen that in this 
example, any two distributions $G_1, G_2$  with the same  mean cannot be distinguished. Consequently, regardless of the 
sample size, the true distribution $G$ cannot be consistently estimated.

\subsection{Empirical Bayes: a brief review.}
The idea of Empirical Bayes was suggested by Robbins (see  Robbins 1953, 1956, 1964).
In   Parametric Empirical Bayes, the prior (or mixing distribution) $G$ is assumed to belong to a parametric family of distributions, and the task
is to estimate $G$ based on the observed $Y_1,...,Y_n$.
For example, consider the case  $$f(y|\theta)=N(\theta,1), \; \theta \in \Theta \subseteq \R,$$ 
 with $\{G\}=\{N(0,\sigma^2), \; \sigma^2 \in \R_+ \}$, where $\sigma^2$ is unknown. 

 In Nonparametric 
Empirical Bayes  the prior $G$ is allowed to be any distribution on the parameter set $\Theta$.

Two main approaches exist: 

{\bf 1. }  {\bf  $f$-modelling approach:} This approach estimates $f_{G}(y)$-the marginal density of $Y$, indirectly without estimating $G$. 

{\bf 2.} {\bf   $G$-modelling  approach:} This approach estimates $G$ directly. 

See Efron (2014) for 
further discussion  of these two approaches. 
 For reviews on Empirical Bayes and its applications see
 e.g.,  Efron ( 2010), and   
Zhang ( 2003).

In nonparametric  $G$-modelling, $G$ is typically estimated using GMLE. Traditionally 
the computation of  GMLE was carried out via the EM-algorithm, see Laird (1978). More recently, Koenker and Mizera (2014) proposed modern convex  optimization techniques  for  computation.

{\bf Estimating the individual parameters $\boldsymbol{\theta}_i$.}
The more  common task in Empirical Bayes is the point estimation of the individual parameters $\theta_i$,  where $Y_i \sim f(y|\theta_i)$, $i=1,...,n$. However, in this paper, we emphasize the task of estimating the mean of 
various functions $\eta(\theta)$ under the mixture $G$. Indeed, the more common task is problematic, without an identifiability assumption.

NPEB estimation of the individual parameter $\theta_i$, $i=1,...,n$, based on GMLE $\hat{G}$, is through:
$$\hat{\theta}_i= E_{\hat{G}} (\theta_i| Y_i).$$

The difficulty under non-unique and inconsistent $\hat{G}$ is illustrated  with Example \ref{ex:first}.
Consider $\hat{G}_1$ and $\hat{G}_2$ in Example \ref{ex:first}. 
 Then: 
$$0.5= E_{\hat{G}_1} (\theta_i|Y_i=1) \neq E_{\hat{G}_2} (\theta_i|Y_i=1)=0.625.$$
making it ambiguous  which estimate should be preferred.
The above ambiguity remains regardless of how large  $n$ is.

It is worth noting that  Greenshtein and Ritov (2022), 
show the following identity: 
$$\eta_{\hat{G}}=\frac{1}{n} \sum_j E_{\hat{G}}(\theta_j|Y_j),$$ for any GMLE $\hat{G}$.

\subsection{Overview and further preliminaries} \label{sec:over}


To illustrate the key ideas, consider a simple scenario. Suppose the goal is to estimate the proportion
of unemployed individuals,
based on  a sample of  size $n$. 
Define: \newline

\hspace{0.2 in} $Y_i=1$ if individual $i$ is unemployed, $Y_i=0$ otherwise. \newline
The obvious estimator is $ (\sum Y_i)/n$. However, this estimator becomes  problematic if there are some nonresponse cases
(or missing values), particularly if  individuals that did not respond  are MNAR.
 For instance, unemployed individuals are more likely   not to respond.

Another example: Suppose $Y_i \sim f(y|\theta_i)$.  Let $I_i$  the indicator of the event:
``a response was obtained from individual $i$", let $h(Y_i)$ denote, for example,  the salary associated with 
individual  $i$.

The goal is to estimate 
$$E_{G} h(Y) =E_G E(h(Y)| \theta)=E_G \eta(\theta) \equiv \eta_G,$$
where $$ \eta(\theta)=E(h(Y)|\theta).$$ 
When there is a positive correlation, under $G$, between $I$ and $h(Y)$, a positive bias, 
of the obvious estimator, is expected.
Specifically, suppose $\sum I_i>0$; then  the naive estimator $$\frac {  \sum_i h(Y_i) I_i}{\sum I_i},$$
is clearly biased for $E_G h(Y)$.

Suppose that under the mixture $G$, $E(h(Y)|I) \perp I| \theta$,
i.e., $E_G (h(Y)|I)$ is independent of $I$ conditional on $\theta$. 
We obtain:
 $$E_G(h(Y)|I=1,\theta)=E_G(h(Y)|I=0,\theta)=E_G(h(Y)|\theta).$$
The above motivates Part i)  of Assumption \ref{as: cond}.
Consider  the following  motivating examples in  
Subsection \ref{sec:ex}. 
In those examples  we have independence of  $E_G(h(Y)|I)$ and $I$ conditional on $\theta$.




Neutralizing the correlation when conditioning on $\theta_i$, aligns with the principle  of Missing At Random (MAR),
as discussed  in Little and Rubin  (2002), where conditioning is performed  on known strata. Our situation is more complex since the ``strata" (parameters)
are unknown and unobserved. 


{\bf  Structure of the paper.}
In Section \ref{sec:uniqueness}, we establish conditions under which, even though $\hat{G}$
is not unique, $\eta_{\hat{G}}$ is unique and does not depend on the particular GMLE
$\hat{G}$.  Section \ref{sec:weak} proves, under appropriate conditions, that for   every sequence 
$\hat{G}$  of GMLE, $\eta_{\hat{G}} \rightarrow_{a.s.}  \eta_G$.  Section \ref{sec:num} presents  numerical experiments.
Section \ref{sec:weighted}  discusses preliminary ideas
for estimating  weighted averages in more complex scenarios.

\subsection{Examples.} \label{sec:ex}

The next example is studied in Greenshtein and Itskov (2014), (GI).

\begin{example} \label{ex:GI}
Consider a survey, in which each sampled item is approached up to $K$
times. If no response is obtained within $K$ attempts the outcome is recorded as ``nonresponse". The  response  $X$
takes the values, $X=x_1,...,x_S$, e.g., $x_1=\mbox{"employed"}, \; x_2=\mbox{ "unemployed"}$, as in (GI). 
Let $K_i$  be the number of attempts until a response, let $Y_i= (X_i, K_i)$. However, we observe
$t(Y_i)$. Here, $t(Y_i)=(X_i,K_i)$  if  $K_i \leq K$; otherwise,   $t(Y_i)=\mbox{``nonresponse"}$.
The possible observed  values  $t(Y_i)$ are $(X_i, K_i)$ and "Nonresponse", leading to
$K \times S +1$ possible  observed values in total (including nonresponse).

Each sampled item $i$ has  a random   response  governed by 
a corresponding parameter $\theta_i$, $i=1,...,n$. In the above example consider the observed number of attempts $K_i$ 
as a truncated Geometric random variable, with parameter $\pi^i$. The probability of $X_i=x_s$ under $\theta_i$ is $p^i_s, \; s=1,2,...,S$, where $\theta_i=(\pi^i, p^1_i,...,p_S^i)$. Assume that conditional on $\theta_i$, $K_i$ is independent of $X_i$.
One may further extend the model to allow $\pi^i=(\pi_1^i,...,\pi^i_K)$,
where $\pi^i_k$ is the probability under $\theta_i$, of item $i$ responding at the $k'th$ attempt.

Under this setup, the goal is to estimate $E_G \eta(\theta)$ where $\eta(\theta_i)=(p^i_1,...,p^i_S)$. This provides an estimate of the population  proportion of items with $X=x_s, \; s=1,...,S$, e.g., the proportion of unemployed individuals in the population (assuming  the original sample including the non-response items
is representative).

\end{example}
The following example is studied in Greenshtein and Ritov (2022), (GR).

\begin{example} \label{ex:GR}
a) 
To estimate the population's proportion of (say) unemployed, a sample  is designed
in $n$ ``small areas/ strata".  Strata are chosen to be ``small" to approximate
the ``Missing At Random"  assumption conditional on a stratum. 
A sample of size $\kappa$ is taken from each of the $n$ strata. Let $\kappa_i$ be the number of responses 
in stratum $i$, $i=1,...,n$. The possible  observed outcomes are $(X_i, \kappa_i)$, $i=1,...,n$, where $X_i$ is the number of unemployed among the $\kappa_i$ responders, while $\kappa_i=0,...,\kappa,$ is the number of responders in stratum $i$.

The possible observed outcomes are $(X_i,\kappa_i)$, where  $X_i$ is the number of unemployed individuals among the 
$\kappa_i$ responders, and $\kappa_i=0,...,\kappa$ represents the number  of responders in stratum $i$.

 Assume  the conditional 
distribution of $X_i$  given $\kappa_i$ is binomial $B(\kappa_i, p_i)$, while $\kappa_i \sim B(\kappa,  \pi_i)$.
 
Define  $\theta_i=(\pi_i,p_i)$,
and $\eta(\theta_i)= p_i$.  The population's unemployed proportion equals $E_G \eta(\theta)$   (assuming strata are of equal size) .

The difficulty  arises when, for some strata $i$, $\kappa_i=0$. In such cases, the obvious estimator 
$\frac{1}{n} \sum_i \frac{X_i}{\kappa_i}$ cannot be applied. 

In terms of  the formulation in the previous subsection think of $\kappa$ as a  truncation of a random variable $S_i \geq \kappa$,
where  samples in stratum $i$ proceeds $S_i$ times until at least one  response is obtained. In particular if
$\kappa_i>0$, $S_i=\kappa$. Let $\tilde{\kappa}_i$ be the number responses among the sampled $S_i$
 $\tilde{X}_i$ the number of unemployed  among $S_i$. Let $Y_i=(\tilde{X}_i, \tilde{\kappa}_i)$. Let 
$t(Y_i)=Y_i$  if $S_i \leq \kappa$; otherwise, $t(Y_i)={\cal I}(S_i>\kappa)$, $\cal{I}$ is an indicator.

b) A similar setup is suitable in observational studies, where $\kappa_i$,   the number of observed individuals in 
stratum $i$ within one unit of time, follows a  $Poisson(\lambda_i)$ distribution. This setup occurs, for example, when  estimating
the spread of a disease, based on the proportion, $p_i$, of infected individuals in each stratum, $i$, $i=1,...,n$.

The estimator is  based on a ``convenience sample" from small strata, of people who underwent
(possibly unrelated) blood tests.
The available data  include the number of tested  individuals  in each stratum and the corresponding test results.
Strata are chosen  to be ``small"  so  that the sampled/tested
individuals in a stratum, are  reasonably representative of non-sampled individuals.

Here the parameter of interest is 
$$\theta_i=(\lambda_i,p_i),  \; \eta(\theta_i)=p_i.$$

As before, a difficulty arises when $\kappa_i=0$ for some $i$.

Suppose in each stratum the time of collecting data is one unit. As in the previous example, we may consider
 collecting data for $T_i \geq 1$ units until  the number  of observed items in stratum $i$ is at least one. Again  we observe only
the data that reached until $T_i=1$. 

Note, in the  current case, where $\kappa_i$ is distributed  Poisson, (GR) proved  that $\hat{G}^n \Rightarrow G$. Thus, the good simulation and real data analysis results, reported in (GR), are less surprising.

\end{example}

\bigskip
 \section{On the  Asymptotic Uniqueness of ${\eta}_{\hat{G}}$}. \label{sec:uniqueness}

In the remainder of the paper $\hat{G}=\hat{G}^n$ denotes  a GMLE based on $t(Y_1),...,t(Y_n)$. When $\hat{G}$
is not unique, we refer to {\it any} GMLE, and any corresponding estimator 
$\eta_{\hat{G}}$.

\subsection{Uniqueness of  $f_{\hat{G}} (y_i)$ for the observed $y_i$.} 

To simplify notations, in this subsection we  do not distinguish between $Y_i$ and $t(Y_i)$, and
$Y_i$ denotes observation $i$.

Given observed $Y_1=y_1,..., Y_n=y_n$, let $\hat{G}_1$ and $\hat{G}_2$ be two GMLEs. Then:

\begin{lemma} \label{lem:main}
For each   $y_i$, $i=1,...,n$, 
\begin{equation} \label{eq:lem1}
f_{\hat{G}_1}(y_i)=f_{\hat{G}_2}(y_i).
\end{equation}
\end{lemma}

We also consider missing values and non-response.
Let $A^c$ the set of all  values $y$  for which the outcome is non-response; i.e., non-response occurs 
  if and only if $y \in  A^c$. 
We define: \newline
when 
\begin{equation} \label{eq:lem11}
y \in A^c,  \; \; f_{\hat{G}} (y) \equiv P_{\hat{G}} (A^c).
\end{equation}


A reviewer commented that Lemma \ref{lem:main}  is related to  Theorem 18  of Lindsay (1995). This is true; however, our proof
requires weaker assumptions  and  appears simpler and more directly related to the present paper.

Before proving the above we need the following preparations.
Recall:  
\begin{equation} \label{eqn:L}
\hat{G}=\argmax_{G} \Pi_i  f_G(Y_i) \equiv \argmax_{G} \log( \Pi_i  f_G(Y_i) ) 
\equiv \argmax_G \log(L(G)).
\end{equation}

The functional $\log(L(G))$ is concave but not necessarily strictly concave. Thus, given two maximizers, i.e., two GMLEs $\hat{G}_1, \hat{G}_2$, for any 
$\lambda \in (0,1)$ also $\lambda\hat{G}_1 +(1-\lambda)\hat{G}_2 \equiv  G_\lambda$ is a GMLE.

Denote by $ g_i= \frac{ d\hat{G}_i}{d\nu}$ the density of $\hat{G}_i$, with respect to some dominating measure $\nu$,  $ i=1,2$. 
 We obtain that: for {\it any} $\lambda \in (0,1)$
 
\begin{eqnarray}
\frac{d}{d\lambda} \log(  L(G_\lambda)) &=& \frac{d}{d\lambda}\sum_i   \log [\int (  f(y_i|\vartheta) [\lambda g_1(\vartheta)+(1-\lambda) g_2(\vartheta)] d\nu(\vartheta) ] \nonumber \\
&=& \sum_i 
\frac{  f_{\hat{G}_1}(Y_i) -  f_{\hat{G}_2}(Y_i)  } 
{ \lambda f_{\hat{G}_1}(Y_i) + (1-\lambda)  f_{\hat{G}_2}(Y_i)  }=0.  
\end{eqnarray}



Denote $c_i=\frac{ f_{\hat{G}_1}(Y_i)  }{ f_{\hat{G}_2}(Y_i) }  $.

Suppose  the above derivative is zero at $\lambda_0 \in (0,1)$. It may be verified that, unless $c_i \equiv 1$, it cannot  be zero
at $\tilde{\lambda}=\lambda_0+ \epsilon$, for sufficiently small $\epsilon >0$  so  that $\tilde{\lambda} \in (0,1)$. This 
is evident by splitting the above sum into two parts: one where $c_i \leq 1$ and another where $c_i>1$.

 \begin{proof}
The proof follows since $c_i \equiv 1$ implies  $f_{\hat{G}_1}(Y_i)  = f_{\hat{G}_2} (Y_i), \; i=1,...,n.$
\end{proof}

\subsection{Assumptions  }\label{sec:unique}


In the following, the set $A$ is abstract and general. However, it is motivated by nonresponse 
setups.

Given the response indicator $I$, define the set $A$ by
$$I \equiv I(y) \equiv I_A(y)=1 \Leftrightarrow y \in A.$$

  Let $\eta(\theta)$ be the function of interest; the goal is to estimate $E_G \eta(\theta)\equiv \eta_G$.
\begin{Assumption}\label{as: cond}
 Given a set $A$, there exists a function $h$, such that under any mixture distribution
$G^*$ on the parameter set $\Theta$: 

i) $\eta(\theta)= E_{G^*} (h(Y)|I=1,\theta)$ for every $\theta \in \Theta$.

ii) $\inf_{y \in A, \theta}  f(y|\theta)> 0.$

iii)  
a)  $f(y|\theta), \;y \in A, \; \theta \in \Theta$ 
is uniformly continuous as a bivariate function of  
$y$ and $\theta$.    

\hspace{0.2 in} b) $\sup_{y \in A,\theta} f(y| \theta)
< \infty$. 

\hspace{0.22 in} c) The function
$h(y)$
 is  bounded and continuous.

iv) There exists an integrable function under $\mu$, denoted $g$, such that $g(y)$ dominates all the functions
of $y$, in the collection $\{ f(y|\theta) \; \theta \in \Theta\}$, that is, for every $\theta \in \Theta$ and every 
$y \in A$, $g(y) \geq f(y|\theta)$. 
\end{Assumption}

Part i) in the above assumption is in the spirit of the MAR condition, conditional on $\theta$; see 
Little and Rubin (2002). Unlike 
the setup in MAR, we are able to draw inference even if, for many $\theta_i$, we have no
corresponding  observation, $Y_i$, satisfying $Y_i \in A$.
 Part ii)  and Part iii) ensure a bounded conditional density  when conditioning on $A$. 
In addition, Part (iii) implies continuity and boundedness of $\eta{\theta}$; thus it implies
convergence  of $E_{\hat{G}} \eta(\theta)$ to $E_G \eta(\theta)$ when $\hat{G}$ converges weakly to  $G$.  It also trivially
ensures   the integrability of $h(y)$. Uniform continuity is needed in the proof of Theorem  \ref{thm:main2}.
Part iv) is needed for the proof of Theorem \ref{thm:main2}.

In both of our motivating examples (Examples \ref{ex:GI} and \ref{ex:GR}), the above assumption is satisfied.
In the first case, take $A$ to be the set of all $y=(x,k)$ such  that $k \leq K$. In the second case. 
take $A$ to be the set of all 
$y=(x,{\kappa^*})$ such that $\kappa^*>0$. 

$\bullet$ In the first example, for $y \in A$, let  the corresponding $h$ be 
$$h(y)=(I(X=x_1),...,I(X=x_S)).$$

$\bullet$ In the second example, for $y \in A$, let $$h(y)=\frac{x}{\kappa^*}.$$   


 \subsection{The truncated and censored cases.}
 
For simplicity, in this subsection we impose  the following Assumption \ref{as:B}. 
(It will be removed in later subsections). 

\begin {Assumption} \label{as:B}
Given the realized observations $t(Y_1),...,t(Y_n)$, for every $y \in A$ there exists a corresponding $j=1,...,n$,
such that $t(Y_j)=y$.
\end{Assumption}

{\bf Truncated case}

In the  truncated  setup,  we do not know about observations $Y=y$  such that $y \notin A$.   As examples, consider the well-known models and problems, ``Capture- recapture",
``How many words did Shakespeare know?", ``Number of unseen  species". In those examples, we do not know about 
unobserved items.

 In the truncated setup, the corresponding densities are:
$$ \{ f^T (y|\theta) \equiv  \frac{f(y|\theta)}{E(I|\theta) }\times I(y\in A), \; \theta \in \Theta \}.$$

{\bf Censored case}

In the case of censored non-response,  for items that did not respond, we do know that their corresponding value  $y$ 
satisfies
$y \notin A$. 
Here, we consider GMLEs that also incorporate the censored information, i.e., for $y_i \notin A$
we utilize the information $y_i \in A^c$.

The corresponding  density  is:
$$ f^C(y|\theta)= f(y|\theta) \;\mbox{if} \; y \in A   \; ; \; f^C(y| \theta)= {P_\theta(A^c)} \; \mbox{if} \; y \in A^c.$$

\bigskip

Let  $\hat{G}^T$ and $\hat{G}^C$  be the GMLEs that are
obtained in the truncated  model and in the censored model respectively;
we write
 $\hat{G}^M$,  $M=T,C$. Similarly, denote by $f^M_{\hat{G}}$, $M=T,C$, the marginal densities of $Y$ under 
 the truncated and the censored cases respectively; finally
we write $f^M(y|\theta)$, $M=T,C$.

\begin{lemma} \label{lem:prop}
Under Assumptions \ref{as: cond}, \ref{as:B},
in the truncated
case  \newline
for any two GMLEs $\hat{G}_1^T, \hat{G}_2^T$, 
$$\eta_{\hat{G}_1^1}=\eta_{\hat{G}_2^T}.$$
\end{lemma}
\begin{proof}
For  $\hat{G}^T_j$, $j=1,2$ we have:

\begin{eqnarray}
\eta_{\hat{G}^T_j}&=& \int \eta(\vartheta) d\hat{G}^T_j(\vartheta)\\
                        &=& \int_\Theta \int_A h(y) \frac{ f(y|\vartheta) }{E(I|\vartheta) } d\mu d\hat{G}_j^T(\vartheta) \\
                          &=& \int_A \int_\Theta f^T(y|\vartheta) d\hat{G}^T_j(\vartheta)d\mu \\
                           &=& \int_A f_{\hat{G}^T_j}^T(y) d\mu 
\end{eqnarray}

The second equality above is by Part i) of Assumption \ref{as: cond}.
The third and fourth equalities are obtained  by interchanging the order of integration and by the definition of $f_{\hat{G}^T}$.
By Lemma \ref{lem:main} and Assumption \ref{as:B}, for every $y \in A$,
 $f_{\hat{G}^T_1}^T(y)=f_{\hat{G}_2^T}^T(y)$. The proof now follows.
\end{proof}

We are not able to prove the analog of the last lemma for the censored case. 

In the next subsection we will obtain an extension of the above result, without relying on Assumption \ref{as:B}.

\subsection{Finite support the truncated case.}

In the remainder of this section, we will  concentrate on the truncated case. The reason is that  the
established results depend
on Lemma \ref{lem:prop} that is proved for the truncated case.  We will consider the censored case
in Subsection \ref{sec:cens}


Consider setups where $A$ has  a finite support, as described  in Section \ref{sec:ex}. 
  By Part ii) of Assumption \ref{as: cond}, for sufficiently large $n$, $n=1,2,...$,  every possible outcome $y \in A$ is realized almost surely. Therefore, almost surely, Assumption \ref{as:B} is satisfied for the observed $t(Y_1),...,t(Y_n)$, for large enough $n$.


We obtain:
\begin{theorem} \label{thm:main1}
Consider the truncated case, assume the set $A$ is finite.
 Under  Assumption \ref{as: cond}, for any  two  sequences of GMLEs $\hat{G}^n_1$ and $\hat{G}_2^n$, 
almost surely
$$\lim_{n \rightarrow \infty } [E_{\hat{G}_1^n} \eta(\theta) -  E_{\hat{G}_2^n}\eta(\theta)]=0.$$
\end{theorem}


\begin{proof}  

 The proof follows the same reasoning as  Lemma \ref{lem:prop}, 
since Assumption \ref{as:B} is satisfied with probability 1 for  large enough $n$.

\end{proof}
\bigskip
{\bf Remark:} Note  that in the setup of Example \ref{ex:first},  for the case $\eta(\theta)=\theta^2$, the equality  stated in the above theorem does not hold. This is because Part i) of Assumption \ref{as: cond}   does not hold.
In that example the set $A$ is the  entire sample space.

\subsection{Infinite support, truncated  case.}

Recall $\hat{G} \equiv \hat{G}^n$, and we omit the superscript $n$.

\begin{theorem} \label {thm:main2}
Consider the truncated case.
Under Assumption  \ref{as: cond},  for {\it any} support size and any two sequences of GMLEs, $\hat{G}_1^T$ and $\hat{G}_2^T$  almost surely:
$$\lim_{n \rightarrow \infty}[ E_{\hat{G}_1^T} \eta(\theta) -   E_{\hat{G}_2^T} \eta(\theta) ]=0.$$

\end{theorem}

\begin{proof}
Consider  the truncated case.
Without loss of generality (w.l.o.g.), we may assume that $A$ is compact. This is because, for  any integrable function
the integral over an entire domain can be approximated arbitrarily closely by the integral over a  sufficiently 
large compact subset
of that domain. In particular, this holds for $\int_A h(y)Kg(y) d\mu$, where $g$ is as in Part iv) of Assumption \ref{as: cond}, and $K= 1/(\inf_\theta P_\theta(A))$.  By Part ii) of  Assumption \ref{as: cond}, $K<\infty$. 
Now, by dominance of $g$ the same approximating compact subset that is obtained under $Kg$, also approximates the integral
above, when  replacing $Kg$ by ${f^T_{\hat{G}^n}(y) }$.

 Let   $$\psi_{\hat{G}}(y)= h(y) f_{\hat{G}}^T(y).$$

By the uniform continuity and boundedness assumed in Assumption \ref{as: cond},
for every $\epsilon>0$ and every $\hat{G}\equiv \hat{G}^n$ there exists $\delta>0$ such that $||y-y'|| < \delta$ implies  
$|\psi_{\hat{G}^T}(y)-\psi_{\hat{G}^T}(y')| <\epsilon$, for every $n$. 
Thus, since $A$ is compact w.l.o.g., there exists a finite collection  of  sets $\Delta_1,..., \Delta_K$, 
with diameter $\delta$ and measure $\mu(\Delta_k)>0$, whose union contains $A$, such that:
if $y$ and $y'$ are in $\Delta_k$, $k=1,...,K$, then $|\psi_{\hat{G}_j^T}(y)-\psi_{\hat{G}_j^T}(y')|< \epsilon$. 
Given  $\Delta_1,...,\Delta_K$, and the true $G$,  almost surely for large enough $n$, for every $1 \leq k  \leq K$ there will be at least one realization $Y_j$, such that $t(Y_j) \in \Delta_k$.

Note that for every $j$ as above, by Lemma \ref{lem:main},  
$$\psi_{\hat{G}_1}^T(t(Y_j))= \psi_{\hat{G}_2}^T(t(Y_j)).$$

It then follows that for each $k$ and $y \in \Delta_k$ almost surely,  
$$|\psi_{\hat{G}_1}^T(y)- \psi_{\hat{G}_2}^T(y)|< 2\epsilon,$$ for any two GMLEs $\hat{G}^T_1$, $\hat{G}^T_2$,
and large enogh $n$.
From  the above we conclude that, for sufficiently large $n$ almost surely:
 $$|\eta_{\hat{G}_1}^T-\eta_{\hat{G}_2}^T| < 2\cal{K}\epsilon;$$
here ${\cal K}= \int_A 1 d\mu<\infty$ by Part ii) of Assumption \ref{as: cond}. 
The above may achieved for every $\epsilon>0$, hence the proof follows.


\end{proof}

\section{ Convergence of $\hat{\eta}_G$ to $\eta_G$, the truncated case.}\label{sec:weak}


In Kiefer and Wolfowitz (1956) and  in Chen (2017), conditions are provided under which,   $\hat{G} \equiv \hat{G}^n$ converges weakly to the true $G$ almost surely, i.e.,
$\hat{G} \Rightarrow G$. These conditions involve the identifiability assumption.

Recall  such a weak convergence would imply ${\eta}_{\hat{G}} \rightarrow_{a.s.} \eta_G$, provided that
$\eta(\theta)$ is bounded and continuous.

Our requirements in this section, which ensure
almost sure convergence of $\eta_{\hat{G}}$ to $\eta_G$, are weaker than those in the aforementioned papers. We only require the almost sure 
 {\it existence} of  a sequence $\hat{G}$ that  converges weakly to
the true mixing  distribution $G$. In particular, we do not assume identifiability.

 Define $\Gamma^n= \{ \hat{G}^n| \hat{G}^n \mbox{is a GMLE} \} \label{eqn: gamma}$.

\begin{Assumption}\label{as: 2}
Consider the truncated case.
Assume  that  there {\it exists}
a sequence $\hat{G}^n_0$, where $\hat{G}^n_0 \in \Gamma^n$ such that $\hat{G}^n_0 \Rightarrow_w G$,
where $G$ is the true mixing distribution.
\end{Assumption}  

The  reasoning and corresponding requirements for a.s.  weak convergence of
 {\it all} GMLE sequences $\hat{G}^n$ to $G$, 
are described below. In addition, we explain why our  Assumption \ref{as: 2} is weaker.

For every distribution $H$ on $\Theta$, 
by the law of large numbers:

i) $\frac{1}{n}\sum_i \log (f_H(Y_i)) \rightarrow   \; E_G \log( f_H(Y))$.

By definition:

ii) $\max_{H \in \{G\}} \frac{1}{n} \sum \log( f_H(Y_i))= \frac{1}{n} \sum \log(f_{\hat{G}}(Y_i)).$

Using a standard argument based on  Jensen's inequality:

iii) $\max_{H \in \{G\} } E_G \log (f_H(Y))= E_G \log( f_G(Y)).$

The above, together with appropriate bounds on $\{ \log(f_H(Y)),\; H \in \{ G \} \}$, the concavity of
$\log(L)$, and an {\it identifiability}
condition, ensure a.s. weak convergence of any sequence $\hat{G}^n$ of GMLEs to the true $G$.

However, under  nonidentifiability,  $\argmax_H$ in ii) and iii),  may not be unique.
Therefore, to guarantee convergence of every GMLE sequence,  $\hat{G}^n$, to the true $G$, identifiability is essential. 
(see also Greenshtein and Ritov (2022), Theorem 4). 
 
To clarify, consider Example \ref{ex:first}. Suppose the true mixing distribution $G$  is $U[0,1]$.
As $n \rightarrow \infty$, 
$$\bar{Y}_n=(\sum_i Y_i/n) \rightarrow_{a.s.} 0.5.$$
 It may be verified that:
$$\Gamma^n=\{ H | E_H \theta=\bar{Y}_n \},$$ 
where $H$ is a distribution on the interval $[0,1]$. In particular, a.s.
there exists a sequence $\hat{G}_0^n \in \Gamma^n$ that converges weakly to  the true distribution $G=U[0,1]$.
 However,
there also exists another sequence $\hat{G}_1^n$ that converges (say) to the distribution $G'$, that satisfies
 $$P_{G'}( \theta=0.5)=1.$$ 

We have demonstrated that not every sequence of GMLE converges weakly to the true $G$;
this is due to the  non-identifiability in the last example.

We now summarize the key points.
In the aforementioned   papers, conditions are provided to ensure that 
almost every sequence of GMLEs  converges weakly to the true $G$. For such a result, the identifiability assumption is crucial.
In contrast, for our consistency result in this section, we  require only that  there exists of a GMLE sequence $\hat{G}^n$ that converges
weakly 
to the true $G$. This condition   is often satisfied without an identifiability assumption, making it a weaker requirement.

\begin{theorem}\label{thm:main3}
Consider  the truncated case.
Under  Assumption \ref{as: cond} and Assumption
 \ref{as: 2}, it follows that for every
sequence $\hat{G}^n$, $n=1,2,...$, $\hat{G}^n \in \Gamma^n$
$$\hat{\eta}_{\hat{G}^n}  \rightarrow_{a.s.} \eta_G.$$
\end{theorem}

\begin{proof}
 Since $\eta(\theta)$ is assumed bounded and continuous, weak convergence implies $\eta_{\hat{G}^n_0} \rightarrow_{a.s.} \eta_G$.
The proof follows by Theorem \ref{thm:main2}, as for any sequence $\hat{G}^n \in \Gamma^n$,  $\eta_{\hat{G}^n_0 } - \eta_{\hat{G}^n} \rightarrow_{a.s.} 0$. 
\end{proof}

\bigskip


\subsection{The censored cae.}\label{sec:cens}

There is a somewhat awkward way to treat the censored case and  obtain an analogous result to
Theorem \ref{thm:main3}. This is by a reduction of  the  censored case  to the truncated case
simply by ignoring the censored information.

\bigskip

The following is a heuristics argument to explain why we may expect that also for certain  $\hat{G}^C$,
$\eta_{\hat{G}^C}$  should asymptotically approximate $\eta_{G}$ well. Consider an EM algorithm that starts
with a prior (or initial guess) as any $\hat{G}^T_0$.  It will converge to  a  mixture $\hat{G}^C_0$ with a minimal 
KL distance between $f_{\hat{G}^T_0}$ and  $\{f_{\hat{G}^C}, \hat{G}^C \mbox {is a GMLE} \} $,
Now,  since $\eta_{\hat{G}^T_0}$   approximates $\eta_G$ well, we might expect the same from 
$\eta_{\hat{G}^C_0}$.

\section{Numerical Example} \label{sec:num}

The  following examples  are taken directly from Greenshtein and Ritov (2022),  including our remark about our surprise
(at the time)  regarding
the good performance of the estimators ${\eta}_{\hat{G}}$. The reported simulations  in the following tables are
based on averages of 50 repetitions in each configuration. The GMLE was computed under the censored  setup.

We study the  model  described in Part a) of Example \ref{ex:GR}.
In the first set of simulations, the attempted sample size is $\kappa=4$,
while  $\kappa_i$ is the realized sample size from stratum $i$, 
 $\kappa_i \sim B(\kappa,\pi_i)$. Our simulated populations  have two types of strata, with
500 strata of each type. For 500 strata  $\pi_i=p_i=0.5-\del$, while for the other 500 strata, $\pi_i=p_i=0.5+\del$.
 The  simulation
 results are reported in Table \ref{tab:Binom1}.

 Note,   $\eta_G=0.5$ throughout the three configurations.  
In the table we  present the mean  and the standard deviation of 
the following estimators. The  Naive estimator  is based on 
estimation of $p_i$ by $\frac{X_i}{\kappa_i}$ when $\kappa_i>0$, and ignores the cases where 
$\kappa_i=0$, in a missing-at-random manner. The estimator referred to as GMLE is 
${\eta}_{\hat {G}}$.

\begin{table}
\caption{The mean and standard deviation of two estimators. Binomial Simulation. $\kappa = 4$ and $\pi_i=p_i=0.5\pm \del$. }\label{tab2}
\label{tab:Binom1}
\begin{center}
\begin{tabular}{ |c||c|c| }
 \hline
    $\del$ & Naive & GMLE \\
	\hline \hline	
 $0.3$ & {\bf 0.559}, (0.012) & {\bf 0.502}, (0.014) \\
 $0.2$ & {\bf 0.522}, (0.011) & {\bf 0.504}, (0.012) \\
 $0.1$ & {\bf 0.504}, (0.010) & {\bf 0.501}, (0.010) \\
 \hline
\end{tabular}
\end{center}
\end{table}

The final set of  simulations is summarized in Table \ref{tab:Binomial2}, where
we study Binomial sampling with various values of $\kappa$, 
$\kappa=1,\dots,5$. Again, the strata are divided evenly into two types:
for 500 strata, 
$p_i$ and $\pi_i$ are sampled independently from  $ U(0.1,0.6)$, while for the remaining 500 strata they are \iid from $U(0.4,0.9)$. Note, throughout  the five configurations $\eta_G=0.5$.

\begin{table}
\caption{Binomial Simulations with continuous $G$. $\kappa$=1,2,3,4,5.}\label{tab3}
\label{tab:Binomial2}
\begin{center}
\begin{tabular}{ |c||c|c| }
 \hline
    $\kappa$ & Naive & GMLE \\
	\hline \hline	
 $1$ & {\bf 0.544}, (0.019) & {\bf 0.530}, (0.015) \\
 $2$ & {\bf 0.528}, (0.014) & {\bf 0.502}, (0.021) \\
 $3$ & {\bf 0.522}, (0.014) & {\bf 0.498}, (0.022) \\
 $4$ & {\bf 0.517}, (0.012) & {\bf 0.499}, (0.020) \\
 $5$ & {\bf 0.512}, (0.009) & {\bf 0.501}, (0.013) \\
 \hline
\end{tabular}
\end{center}
\end{table}

{\it It is surprising how well the GMLE performs well 
already for $\kappa=2,3$, in spite of the
non-identifiability of $G$ and the
inconsistency of the (non-unique) GMLE $\hat{G}$ as an estimator of $G$.

In (GR), the  computation of the GMLE is described. We employed a dense grid of parameters, and applied the EM algorithm. Due to the nonidentifiability,  the algorithm may converge to different GMLEs. In each simulation run, we  used whichever distribution the algorithm converged to.

In the preceding discussion and simulations we concentrated on  cases  where  the observations have a finite number of possible values. Those are interesting cases in the context of this paper,
since nonidentifiability is implied.
Further simulations  and real data examples may be found in the aforementioned papers (GI)  and (GR). In particular in (GR)   real data are analyzed under a Poisson model, where identifiability is implied. In that example the task is to estimate the proportion of households  that own  the unit they live in, based on
a sample of a Poisson size from various Statistical-Areas. The complication is that in some areas the Poisson variable has a value  of 0; thus the ideas in the current paper
were applied.

\section{ Preliminary Ideas on Weighted Averages and  Incorporating Covariates.} \label{sec:weighted}


It may be of  interest to estimate
weighted averages  for example when the strata are not of equal size, and the population average is not a simple average of the strata's averages. 
 Similarly, this  arises when the sampling of strata or of individuals, is not
performed with
equal probabilities.
Suppose the relative weights are $\gamma_1, ... \gamma_n$, and it is desired to estimate $\sum \gamma_i \eta(\theta_i)$, e.g., population average rather than simple average of strata's averages.

We now sketch a plausible approach. The approach is sketched  without claiming
 consistency or any optimality property. 
Following this, we outline how to extend  the ideas of  Empirical Bayes to cases where covariates are present.
Note  that  when covariates are introduced exchangeability is lost, and consequently, the appeal of Empirical Bayes is reduced.
The sketch provided is informed by the previous sections and their corresponding assumptions.

The idea is to introduce additional random variables  and  corresponding parameters that will incorporate the weights
$\gamma_i$.  We demonstrate this through Example 3,
Part a).

Suppose the size of stratum $i$ is $N_i$.  Following a super-population model, we treat $N_i$ as
a realization of a  $Poisson(\lambda _i)$ random variable. 
Hence, we observe $Y_i=(X_i,\kappa_i, N_i)$.

Now, conditional  on $\theta_i=(\pi_i, p_i,\lambda_i)$,
$(X_i,\kappa_i)$ follow the same  distribution as before,  while $N_i \sim Poisson(\lambda_i)$ is independent of 
$(X_i,\kappa_i)$. We define  $\eta(\theta_i)= p_i \lambda_i$. 

The  above approach might be helpful more generally to handle Empirical Bayes problems  involving covariates. 
By treating the covariates as realizations
from a super-population, distributed according to some parametric family, we create 
``apparent exchangeability" among the observations $Y_i$. This, in turn, strengthens the  appeal of the NPEB approach.

\newpage

{\bf \Large References:}

\begin{list}{}{\setlength{\itemindent}{-1em}\setlength{\itemsep}{0.5em}}

\item
Chen, J. (2017).  ``Consistency of the MLE under Mixture Models.'' Statist. Sci. 32 (1) 47 - 63.
\item
Eckles, D., Ignatiadis, N., Wager, S., and Wu, H. (2025). Noise-induced randomization in regression discontinuity designs. Biometrika  112(2).
\item
Dobronyi, C.,  Gu, J., and K. I. Kim (2021). Identification of dynamic panel logit models with fixed effects.  arXiv:2104.04590 (econ)
\item
Efron, B. (2012). Large scale inference: Empirical Bayes  methods for estimation , testing and prediction. Cambridge University Press.
\item
Efron, B. (2014). Two modeling strategies for empirical Bayes estimation. Statistical science: a
review journal of the Institute of Mathematical Statistics, 29(2):285.
\item
Greenshtein, E. and Itskov, T (2018), Application of Nonarametric  Empirical
Bayes to Treatment of Non-Response. {\it Statistica Sinica} 28 (2018), 2189-2208.
\item
Greenshtein, E. and Ritov, Y. (2022). Generalized maximum likelihood estimation of the mean of parameters of mixtures. with applications to sampling and to observational studies. {\it EJS} 16(2): 5934-5954.
\item
Kiefer, J. and Wolfowitz, J. (1956). Consistency of the maximum likelihood estimator in the presence of infinitely many incidental parameters. {\it Ann.Math.Stat.}
27 No. 4, 887-906.
\item
Koenker, R. and Mizera, I. (2014). Convex optimization, shape constraints,
compound decisions and empirical Bayes rules. {\it JASA } 109, 674-685.
\item
Laird, N. (1978). Nonparametric maximum likelihood estimation of a mixing distribution. {\it JASA}  78, No 364, 805-811.
\item
Lindsay, B. G. (1995). Mixture Models: Theory, Geometry and Applications.
Hayward, CA, IMS.
\item
Little, R.J.A and Rubin, D.B. (2002). Statistical Analysis with Missing Data.
New York: Wiley
\item
Robbins, H. (1951). Asymptotically subminimax solutions of compound decision problems. In Proceedings of the Second Berkeley Symposium on Mathematical Statistics and
Probability, 1950 131–148. Univ. California, Berkeley. MR0044803
\item
Robbins, H. (1956). An empirical Bayes approach to statistics. In Proc. Third Berkeley
Symp. 157–164. Univ. California Press, Berkeley. MR0084919
\item
Robbins, H. (1964). The empirical Bayes approach to statistical decision problems. Ann.
Math. Statist. 35 1–20. MR0163407
\item
Robbins, H. and Zhang, C-H. (1991). Estimating multiplicative  treatment effect under biased allocation. Biometrika 78(2) 349-354.
\item
Vardi, Y. (1985). Empirical distribution in selection bias  models.  Ann.Stat. 13  (1), 178-203.
\item
Turnbull, B.  W. (1976). The empirical distribution function with arbitrarily grouped, censored and truncated data.
jrssb (38) 3, 290-295. 
\item
Zhang, C-H. (2003)  Compound decision theory  and empirical Bayes  methods. Ann. Stat.   31 (2),  379-390.

\end{list}

\end{document}